\documentclass[a4paper,reqno]{amsart}
\usepackage{amssymb}
\usepackage{hyperref}
\usepackage{graphicx}
\usepackage{subcaption}
\usepackage{latexsym}
\usepackage{amsmath,amsthm,mathtools}
\usepackage{physics}
\usepackage{euscript}
\usepackage{graphics,color}
\usepackage[all]{xy}
\usepackage{comment}
\usepackage[margin=3cm]{geometry}
\usepackage{MnSymbol} 

\newcommand{\be}{\begin{equation}}
\newcommand{\map}{\zeta}
\newcommand{\id}{I}

\newcommand{\ee}{\end{equation}}
\newcommand{\ba}{\begin{eqnarray}}
\newcommand{\ea}{\end{eqnarray}}
\newcommand{\baa}{\begin{eqnarray*}}
\newcommand{\eaa}{\end{eqnarray*}}
\newcommand{\bb}{\begin{thebibliography}}
\newcommand{\eb}{\end{thebibliography}}

\newcommand{\A}{\mathcal{A}}
\newcommand{\V}{\mathcal{V}}
\newcommand{\U}{\mathcal{U}}

\newcommand{\Vq}{V_q}

\newcommand{\tends}{\rightarrow}

\newcounter{my}
\newcommand{\he}%
   {\stepcounter{equation}\setcounter{my}%
   {\value{equation}}\setcounter{equation}0%
   }%
\newcommand{\she}%
   {\setcounter{equation}{\value{my}}%
    }%

\ExplSyntaxOn

\NewDocumentCommand{\pFq}{O{}mmmmm}
 {
  \group_begin:
  \keys_set:nn { hypergeometric } { #1 }
  \hypergeometric_print:nnnnn { #2 } { #3 } { #4 } { #5 } { #6 }
  \group_end:
 }
\NewDocumentCommand{\hypergeometricsetup}{m}
 {
  \keys_set:nn { hypergeometric } { #1 }
 }

\tl_new:N \l_hypergeometric_divider_tl
\tl_new:N \l_hypergeometric_left_tl
\tl_new:N \l_hypergeometric_right_tl

\keys_define:nn { hypergeometric }
 {
  symbol .tl_set:N = \l_hypergeometric_symbol_tl,
  symbol .initial:n = \phi,
  separator .tl_set:N = \l_hypergeometric_separator_tl,
  separator .initial:n = {},
  skip .tl_set:N = \l_hypergeometric_skip_tl,
  skip .initial:n = 8,
  divider .choice:,
  divider/semicolon .code:n = \tl_set:Nn \l_hypergeometric_divider_tl { \;; },
  divider/bar .code:n = \tl_set:Nn \l_hypergeometric_divider_tl { \;\middle|\; },
  divider .initial:n = bar,
  fences .choice:,
  fences/brack .code:n = 
   \tl_set:Nn \l_hypergeometric_left_tl {[}
   \tl_set:Nn \l_hypergeometric_right_tl {]},
  fences/parens .code:n = 
   \tl_set:Nn \l_hypergeometric_left_tl {(}
   \tl_set:Nn \l_hypergeometric_right_tl {)},
  fences .initial:n = brack,
 }

\cs_new_protected:Nn \hypergeometric_print:nnnnn
 {
  {} \sb {#1} \l_hypergeometric_symbol_tl \sb { #2 }
  \left\l_hypergeometric_left_tl
  \genfrac .. 
           {0pt} 
           {} 
           { \__hypergeometric_process:n { #3 } } 
           { \__hypergeometric_process:n { #4 } } 
  \l_hypergeometric_divider_tl
  #5
  \right\l_hypergeometric_right_tl
 }

\cs_new_protected:Nn \__hypergeometric_process:n
 {
  \clist_use:nn { #1 }
   {
    {\l_hypergeometric_separator_tl}
    \mspace { \l_hypergeometric_skip_tl mu }
   }
 }

\ExplSyntaxOff

\newtheorem{pr}{Proposition}

\newtheorem{theorem}{Theorem}[section]
\newtheorem{proposition}[theorem]{Proposition}
\newtheorem{corollary}[theorem]{Corollary}
\newtheorem{lemma}[theorem]{Lemma}
\newtheorem{definition}[theorem]{Definition}
\theoremstyle{definition}

\numberwithin{equation}{section}

\title{A tale of two $q$-deformations : connecting dual polar spaces and weighted hypercubes}

\author{Pierre-Antoine Bernard}
\address{Centre de Recherches Mathématiques, Université de Montréal, C.P. 6128, Succursale Centre-ville, Montréal, QC H3C 3J7, Canada}
\email{pierre-antoine.bernard@umontreal.ca}

\author{Etienne Poliquin}
\address{Centre de Recherches Mathématiques, Université de Montréal, C.P. 6128, Succursale Centre-ville, Montréal, QC H3C 3J7, Canada}
\email{etienne.poliquin.1@umontreal.ca}

\author{Luc Vinet}
\address{IVADO and Centre de Recherches Mathématiques, Université de Montréal, C.P. 6128, Succursale Centre-ville, Montréal, QC H3C 3J7, Canada}
\email{luc.vinet@umontreal.ca}

\begin{document}

\begin{abstract}
Two $q$-analogs of the hypercube graph are introduced and shown to be related through a graph quotient. The roles of the subspace lattice graph, of a twisted primitive elements of $U_q(\mathfrak{su}(2))$ and of the dual $q$-Krawtchouk polynomials are elaborated upon. This paper is dedicated to Tom Koornwinder.
\end{abstract}

\maketitle

\section{Introduction}

A weighted hypercube was recently introduced by the authors as a $q$-deformation of the usual hypercube graph $Q_N$ \cite{q-hypercube}. The corresponding adjacency matrix $A_q$ projects onto the one-excitation Hamiltonian of the dual $q$-Krawtchouk spin chain when restricted to the subspace spanned by $q$-Dicke states. This extends the known relationship between the hypercube, the Krawtchouk spin chain, and Dicke states to their $q$-deformed counterparts. Moreover, $A_q$ was shown to represent a twisted primitive element of  $U_q(\mathfrak{su}(2))$ within the reducible representation obtained from the $N$-fold tensor product of two-dimensional representations, mirroring the relation between the adjacency matrix of $Q_N$ and the Lie algebra $\mathfrak{su}(2)$.

An alternative approach to $q$-deformed hypercube comes from the dual polar graphs, whose vertex set is composed of the maximal isotropic subspaces of a vector space defined on a finite field $\mathbb{F}_q$, with $q$ a prime power \cite{DRG, stanton1980some}. Like the $N$-cube, they are distance-regular and belong to a family of graphs corresponding to a $P$-polynomial association scheme. While the Hamming scheme to which the hypercube belongs is associated with Krawtchouk polynomials, the distance matrices of the dual polar schemes are linked to dual $q$-Krawtchouk polynomials. Additionally, the Terwilliger algebra of the Hamming and dual polar schemes are known to be respectively isomorphic to
 $\mathfrak{su}(2)$ and $U_q(\mathfrak{su}(2))$ \cite{Bernard_Hamming, Bernard_dual, go2002terwilliger,worawannotai2013dual}.

In this paper, we aim to show that these two frameworks, namely the hypercube with weights and the dual polar graphs, are connected through a graph quotient. Specifically, we will demonstrate that restricting the adjacency matrix of the symplectic dual polar graph to a subspace of its defining module yields the adjacency matrix of weighted hypercubic network as introduced in \cite{q-hypercube}. The structure of the paper is as follows: In Section 2, we review the definition of the weighted hypercubic network that serves as a $q$-deformation of the $N$-cube. In Section 3, we provide an overview of dual polar spaces and their $P$-polynomial association scheme. In Section 4, we demonstrate that the weighted hypercube can be obtained as a quotient graph of the symplectic dual polar graph by leveraging their shared connection to the subspace lattice.

Tom Koornwinder keeps having a profound influence on the field of orthogonal polynomials and special functions, and in particular on their connections with representation theory. He has especially contributed to the development of univariate and multivariate $q$-orthogonal polynomials and their interpretation through quantum algebras and other structures. It was he who stressed the important role in this regard of the twisted primitive element \cite{koornwinder1993askey,koornwinder1993q} which plays a key role in the present study. It is hence with great admiration and best wishes that we offer this paper as a contribution to this volume in his honour.

\section{A first $q$-analog: The weighted hypercube} \label{sec:2}

We recall the definition of the weighted hypercube introduced in \cite{q-hypercube}, and the role of its adjacency matrix $A_q$ as a representation of a twisted primitive element of $U_q(\mathfrak{su}(2))$. 

\subsection{Definition}
Let $V$ be the set of binary sequences of length $N$ and $\partial(x,y)$ denote the Hamming distance between two sequences $x = (x_1, x_2, \dots x_N)$ and $y = (y_1, y_2, \dots y_N)$,
\begin{equation}
    \partial(x,y) = |\{ i \in \{1, 2, \dots, N\} \ |\ x_i \neq y_i\}|.
\end{equation}
The hypercube graph $Q_N$ is defined on the set $V$, with edges connecting sequences $x$ and $y$ at Hamming distance $\partial(x,y) = 1$. For each $x \in V$, we associate an orthonormalized vector $\ket{x} \in \mathbb{C}^{2^N}$. The adjacency matrix of $Q_N$ is given by
\begin{equation}\label{eq:A}
    \bra{x}A\ket{y}=\left\{ \begin{array}{cc}
        1 & \text{if } \partial(x,y) = 1, \\
        0 & \text{otherwise.} 
    \end{array}\right.
\end{equation}
A $q$-deformation of the hypercube is obtained by modifying the adjacency matrix $A$ so that its non-zero entries take on values other than $1$, resulting in a weighted graph:
\begin{equation}\label{def:aq}
    \bra{x}A_q\ket{y}=\left\{ \begin{array}{cc}
          q^{i-N + 2\sum_{j =i+1}^{N} x_j }& \text{if } \partial(x,y) = 1, \ x_{i} \neq y_{i} \\
        0 & \text{otherwise.} 
    \end{array}\right.
\end{equation}
Note that as $q\rightarrow 1$, all weights in equation \eqref{def:aq} converges to $1$,  thereby yielding the adjacency matrix $A$ of the standard hypercube.
\subsection{Relation to $U_q(\mathfrak{su}(2))$}
The matrix $A_q$ has a representation theoretic underpinning in $U_q(\mathfrak{su}(2))$. This Hopf algebra is generated by four elements, denoted $e$, $f$, $k$ and $k^{-1}$, which satisfy the following defining relations,
\begin{gather}\label{uq(sl2)_def}
    k k^{-1}=k^{-1}k=\id,\qquad kek^{-1}=q^2e, \qquad kfk^{-1}=q^{-2}f \\ [e,f]=\frac{k-k^{-1}}{q-q^{-1}}\equiv[h]_q,
\end{gather}
where $k\equiv q^{h}$ and $[x]_q\equiv\frac{q^x-q^{-x}}{q-q^{-1}}$. Its fundamental representation is given by
\begin{equation}
    e\tends \sigma^+ = \begin{pmatrix}
        0&1\\0&0
    \end{pmatrix} \quad f\tends\sigma^-= \begin{pmatrix}
        0&0\\1&0
    \end{pmatrix}, \quad k\tends q^{\sigma^z}=
    \begin{pmatrix}
    q&0\\
    0&q^{-1}
\end{pmatrix}.
\end{equation}
It admits a coproduct $\Delta_q:U_q(\mathfrak{su}(2)) \rightarrow U_q(\mathfrak{su}(2)) \otimes U_q(\mathfrak{su}(2)) $, that is a homomorphism mapping the algebra into its two-fold tensor product as follows,
\begin{gather}\label{eq:copq}
    \Delta_q(f)=f \otimes k^{-1/2}+ k^{1/2} \otimes f, \qquad \Delta_q(e)=e \otimes k^{-1/2} +k^{1/2} \otimes e\\ 
    \Delta_q(k) = k\otimes k, \quad \Delta_q(h) = h\otimes \id + \id \otimes h.
\end{gather}

The matrix $A_q$ can be expressed in terms of the generators of $U_q(\mathfrak{su}(2))$ in the representation formed from $N$ copies of the fundamental representation. Consider the following matrices resulting from $N-1$ application of the coproduct on the $2$-dimensional representation of $e$, $f$ and $k$ :
\begin{equation}
    X^\pm = \Delta_q^{(N-1)}(\sigma^\pm), \quad K = \Delta_q^{(N-1)}(q^{\sigma^z}).
\end{equation}
Using the basis $\ket{x} = \ket{x_1}\otimes \dots \otimes \ket{x_n} \in \mathbb{C}^{2^N}$, with $\ket{0} = \binom{1}{0}$ and $\ket{1} = \binom{0}{1}$, it is straightforward to check that $A_q$ defined in \eqref{def:aq} can be expressed as
\begin{equation} \label{A_q_def}
    A_q = (\sqrt{q} X^- + \frac{1}{\sqrt{q}}X^+)K^{-1/2} = \sum_{i=1}^N \underbrace{I\otimes...\otimes  I}_{i-1\text{ times}}\otimes\ \sigma_x\otimes \underbrace{ q^{-\sigma^z} \otimes...\otimes q^{-\sigma^z}}_{N-i\text{ times}},
\end{equation}
and that it thus represents the element $Y = \left(\sqrt{q} f + \frac{1}{\sqrt{q}} e\right)k ^{-1/2}$ of $U_q(\mathfrak{su}(2))$ in the reducible representation obtained from the $N$-fold tensor product of its fundamental representation.
Note that the element $Y$ is a twisted primitive element of $U_q(\mathfrak{su}(2))$ verifying $\Delta_q(Y) = Y \otimes k^{-1} + I \otimes Y$ and defining a co-ideal subalgebra. 

\section{Second $q$-analog : The dual polar graph}

An alternative $q$-analog of the hypercube originates from the theory of association schemes.  This section introduces the relevant background and the graphs based on dual polar spaces. We begin by recalling the definition of distance-regular graphs and their connection to $P$-polynomial association schemes.

\subsection{Distance-regular graph and association schemes}
A graph is said to be distance-regular if for any pair of vertices $x$ and $y$, the number $p_{ij}^k$ of vertices $z$ at distance $i$ from $x$ and $j$ from $y$ depends only on the distance $k$ between $x$ and $y$. Let $\A_i$ denote the $i$-th distance matrix, whose entries are defined by
\begin{equation} \label{def_A1}
    (\A_i)_{xy} =\left\{
\begin{array}{ll}
      1 & \text{if }\text{dist}(x,y)=i,\\
      0 & \text{otherwise,}
\end{array} 
\right. 
\end{equation}
with $\text{dist}(x,y)$ the distance between the vertices $x$ and $y$ in the graph. The set of distance matrices $\{\A_i \ | \ i = 0, 1, \dots, N \}$ of a distance-regular graph forms a $P$-polynomial association scheme. Specifically, they are $(0,1)$-matrices and satisfy the defining relations of a symmetric association scheme \cite{algebraic_comb, DRG}:
\begin{enumerate}
\item[(i)] $\A_0 = I$, where $I$ is the identity matrix;
\item[(ii)] $\sum_{i = 0}^N \A_i = J$, where $J$ is the matrix of ones and $N$ is the diameter of the graph;
\item[(iii)] $\A_i = \A_i^t$ for all $i \in \{0,1,\dots, N\}$;
\item[(iv)] They verify the so-called Bose-Mesner relations : $\A_i \A_j = \sum_{k=0}^{N}p_{ij}^k \A_k.$
\end{enumerate}
In addition, they verify the $P$-polynomial property which asserts that for each distance $i$ there exists a polynomial $v_i$ of order $i$ such that 
\begin{equation}
    \A_i = v_i(\A_1).
\end{equation}
This condition is equivalent to requiring the Bose-Mesner relations with respect to $\A_1$ to take a three-term form,
\begin{equation}
    \A_1 \A_i = c_{i+1} \A_{i+1} + a_i \A_i + b_{i-1} \A_{i-1}.
\end{equation}

Note that $(0,1)$-matrices $\A_i$ forming a symmetric association scheme are typically referred to as adjacency matrices, as each defines a graph $G_i$. In the case of $P$-polynomial schemes, we use the term distance matrices for $\A_i$ with $i\neq 1$ to emphasize the role of $\A_1$ as the adjacency matrix of a distance-regular graph, with the other $\A_i$ standing for the corresponding distance matrices.

The $N$-cube is a well-known example of a distance-regular graph. Its distance matrices $\A_i$ belong to a set that forms a $P$-polynomial scheme known as the (binary) Hamming scheme $H(N,2)$. Each distance matrix $\A_i$ is expressible in terms of Krawtchouk polynomials $K_i(x; p, N)$ \cite{Koekoek}  evaluated on the adjacency matrix $\A_1 =A$ defined in equation \eqref{eq:A},
\begin{equation}\label{eq:kp}
    \A_i = \binom{N}{i} K_i\left(\frac{N}{2}- \frac{\A_1}{2} ;\frac{1}{2},N \right).
\end{equation}
The Bose-Mesner relation with respect to $\A_1$ reads
\begin{equation}\label{eq:ttrh}
    \A_1 \A_i = (i+1) \A_{i+1} + (N-i+1)\A_{i-1}.
\end{equation}

Next, we review the definitions of the dual Bose-Mesner algebra and of the Terwilliger algebra associated with a symmetric association scheme.

\subsection{The Terwilliger algebra of an association scheme} The definitions of the dual Bose-Mesner and Terwilliger algebras of an association scheme require the introduction of primitive idempotents and dual distance (or adjacency) matrices.  Let $\A_i$ for $i = 0, 1, \dots, N$ be a set of $(0,1)$-matrices that form a symmetric association scheme, and whose rows and columns are labeled by the elements of a set $X$. From the Bose-Mesner relations and the invariance of these matrices under transposition, it follows that they commute, i.e., $[\A_i, \A_j] = 0$, and share a common eigenbasis. Consequently, there exists a set of $N+1$ primitive idempotents $E_i$ that project onto the eigenspaces of the matrices $\A_i$. These idempotents satisfy the following relations :
\begin{equation}
     E_i E_j = \delta_{ij} E_i, \quad \sum_{i  =0 }^N E_i = I,
\end{equation}
and give an alternative basis for the Bose-Mesner algebra. Specifically, there exist coefficients $p_i(j)$ and $q_i(j)$ such that
\begin{equation}\label{eq:cb}
     \A_i = \sum_{i  =0 }^N p_i(j) E_i, \quad E_i = |X|^{-1}\sum_{i  =0 }^N q_i(j) \A_i,
\end{equation}
where 
$|X|$ denotes the cardinality of the set $X$ on which the association scheme is defined. The vector space spanned by the matrices $\A_i$ is closed under entry-wise product $\odot$, as these matrices satisfy $\A_i \odot \A_j = \delta_{ij} \A_i$. Since the matrices $E_i$ span the same space, it follows that there exist coefficients $q_{ij}^k$, known as Krein parameters, such that
\begin{equation}\label{eq:had1}
    E_i \odot E_j = |X|^{-1}\sum_{k  =0 }^N q_{ij}^k E_k.
\end{equation}
The dual matrices $\A_i^*(x_0)$ are \textit{diagonal} matrices defined with respect to a fixed element $x_0 \in X$, with their entries expressed in terms of the entries of the primitive idempotents $E_i$, 
\begin{equation}\label{def:dualA}
    (\A_i^*(x_0))_{xx} = |X| (E_i)_{xx_0}.
\end{equation}
In the following, we denote $\A_i^*$  as shorthand for $ \A_i^*(x_0)$. It can be shown from definition \eqref{def:dualA} and equation \eqref{eq:had1} that these matrices satisfy the so-called \textit{dual Bose-Mesner relations},
\begin{equation}\label{eq:had1}
    \A_i^*  \A_j^* = \sum_{k  =0 }^N q_{ij}^k  \A_k^*.
\end{equation}
The algebra generated by these dual matrices is known as the \textit{dual Bose-Mesner algebra}. Analogous to the concept of a $P$-polynomial association scheme, an association scheme is termed $Q$-polynomial if, for each $i$, the dual matrix $\A_i^*$ can be expressed as a polynomial $v_i^*$ of degree $i$ in $\A_1^*$.
\begin{equation}
    \A_i^* = v_i^* (\A_1^*).
\end{equation}
This is also equivalent to requiring that the dual Bose-Mesner relations involving $\A_1^*$ take a three-term recurrence form,
\begin{equation}
    \A_1^* \A_i^* = c_{i+1}^* \A_{i+1}^* + a_i^* \A_i^* + b_{i-1}^* \A_{i-1}^*.
\end{equation}

The Terwilliger algebra $\mathcal{T}$ of an association scheme is defined as the algebra generated by both the matrices $\A_i$ and their duals $\A_i^*$,
\begin{equation}
    \mathcal{T} = \langle \A_0, \A_1, \dots \A_N, \A_0^*, \A_1^*, \dots \A_N^* \rangle.
\end{equation}
In the case where an association scheme is both $P$- and $Q$-polynomial, the matrices $\A_i$ and $\A_i^*$ can be expressed in terms of $\A_1$ and $\A_1^*$, respectively. This significantly reduces the set of generators for the algebra $\mathcal{T}$, whose definition becomes
\begin{equation}
    \mathcal{T} = \langle \A_1, \A_1^* \rangle.
\end{equation}
This holds for the binary Hamming association scheme, which is known to be both $P$- and $Q$-polynomial. Its matrices $\A_1$ and $\A_1^*$ can be identified with a representation of the generators $2j^x$ and $2j^z$ of $\mathfrak{su}(2)$, establishing a correspondence between $\mathfrak{su}(2)$ and the Terwilliger algebra of the hypercube \cite{Bernard_Hamming, go2002terwilliger}.

\subsection{Dual polar graphs}
Consider a vector space $\mathbb{F}_q^D$ of dimension $D$ defined over a finite field $\mathbb{F}_q$ and equipped with a non-degenerate form $\mathfrak{B}$. A subspace $\V \subset \mathbb{F}_q^D$ is said to be \textit{isotropic} if the form $\mathfrak{B}$ vanishes on any pair of vectors $v_1$ and $v_2$ in $\V$, i.e.
\begin{equation}
    \mathfrak{B}(v_1,v_2) = 0, \quad \forall v_1, v_2 \in \V.
\end{equation}
An isotropic subspace $\V$ is further said to be \textit{maximal} if there is no isotropic subspace $\U$ such that $\V \subset \U$. By Witt's theorem, the maximal isotropic subspaces of $\mathbb{F}_q^D$ all have the same dimension $N \leq D/2$, which is referred to as the Witt index. A \textit{dual polar graph} has for vertices the set $X$ of all maximal isotropic subspaces of the vector space $\mathbb{F}_q^D$ equipped with a form $\mathfrak{B}$. An edge connects two vertices $\V$ and $\U$ if $\text{dim}(\V \cap\, \U) = N-1$. Let each vertex $ \V \in X$ be associated with an orthonormalized vector $\ket{\V}$ in $\mathbb{C}^{|X|}$. The distance matrices $\A_i$ of a dual polar graph are defined by
\begin{equation}
    \bra{\V}\A_i\ket{\,\U} = \left\{
	\begin{array}{ll}
		1  & \mbox{if }  \text{dim}(\V \cap \U) = N-i,\\
		0 & \mbox{otherwise. }
	\end{array}
\right.
\end{equation}
Dual polar graphs and are known to be distance-regular, with their respective distance matrices $\A_i$ expressed as dual $q$-Krawtchouk polynomial $K_i(\lambda(x); c, N | q )$ of $\A_1$ \cite{DRG}:
\begin{equation}\label{eq:dqk}
    \A_i = (-1)^i q^{\binom{i}{2}} \bigl[\begin{smallmatrix}
        N\\i
    \end{smallmatrix}\bigr]_q K_i\left( q^{-N}(1-q) \A_1 + q^{-N}(1-q^e) ; - q^e, N\,|\,q\right),
\end{equation}
where $e$ is a parameter that depends on the type of form $\mathfrak{B}$ and the dimension $D$ (see Table \ref{tab:pc1}) \cite{stanton1980some}. 
\begin{table}[h]
    \centering
    \begin{tabular}{c|c|c|c}
       type of form $\mathfrak{B}$ & dimension $D$ & type of vector space & value of $e$  \\
        \hline
        \hline
         bilinear, skew-symmetric & $2N$ & $C_N(q)$ & $1$  \\
         bilinear, symmetric & $2N$ & $D_N(q)$ &  $1$  \\
        bilinear, symmetric  & $2N + 2$& $ \prescript{2}{}{D}_{N+1}(q) $& $0$  \\
         bilinear, symmetric  & $2N+1$& $B_N(q)$ & $2$  \\
        hermitian & $2N$ & $\prescript{2}{}{A}_{2N}(\sqrt{q})$ & $3/2$  \\
        hermitian & $2N + 1$ & $ \prescript{2}{}{A}_{2N-1}(\sqrt{q}) $ & $1/2$ 
    \end{tabular}
    \caption{Parameter $e$ for the different types of vector spaces of Witt index $N$. The type of a vector space is determined by its dimension  $D$ and the type of the form $\mathfrak{B}$, with notation chosen to align with the classification of Lie-type groups associated with these non-degenerate forms. Note that the last two types require $q$ to be a square of a prime power. }
    \label{tab:pc1}
\end{table}
The Bose-Mesner relations with respect to the adjacency matrix $\A_1$ for an association scheme based on a dual polar graph read \cite{DRG}:
\begin{equation}\label{eq:ttr2}
    \A_1 \A_i = [i+1]_q \A_{i+1} + (q^e -1)[i]_q \A_i + q^{i-1+e}[N-i+1]_q   \A_{i-1}.
\end{equation}

Since equations \eqref{eq:ttrh} and \eqref{eq:kp} can be recovered in the $q\rightarrow1$ limit of equations \eqref{eq:ttr2} and \eqref{eq:dqk} respectively, the dual polar graphs are often referred to in the literature as  distance-regular $q$-analogs of the hypercube. Note that they are composed of $(-q^{N+e-1};q)_N$ vertices \cite{DRG}, which is greater than the number $2^N$ of vertices in the hypercube but converges to the same value as $q$ goes to $1$. This contrasts with the weighted hypercube introduced in the previous section, which has $2^N$ vertices regardless of the value of $q$. This difference motivates looking for a quotient graph relating the two constructions.

Finally, let us note that, like the weighted hypercube discussed in Section \ref{sec:2}, dual polar graphs are also connected to $U_q(\mathfrak{su}(2))$. The association scheme composed of the distance matrices of a dual polar graph is known to be $Q$-polynomial and has a Terwilliger algebra $\mathcal{T}$ generated by the adjacency matrix $\A_1$ and its dual $\A_1^*$. These matrices can be identified with a representation of the generators of $U_{\sqrt{q}}(\mathfrak{su}(2))$. According to Theorem 24.3 of \cite{worawannotai2013dual}, up to an algebra isomorphism, the following holds:

\begin{theorem} (Reformulation of Theorem 24.3 of \cite{worawannotai2013dual}.)
    Let $\A_1$ and $\A_1^*$ denote the adjacency and dual adjacency matrices of the dual polar graph defined on the set $X$ of maximal isotropic subspaces (of dimension $N$) of a vector space defined over a finite field $\mathbb{F}_q$ and of type $e$ with respect to Table \ref{tab:pc1}. Then, there exists a representation $\rho : U_{\sqrt{q}}(\mathfrak{su}(2)) \rightarrow \text{End}(\mathbb{C}^{|X|})$ and two elements $\Upsilon$ and $\Psi$ in the centralizer of $\rho(U_{\sqrt{q}}(\mathfrak{su}(2)))$ such that
    \begin{equation}
    \A_1 = h\rho(1) + ( \kappa \Upsilon^{-1} \Psi + \upsilon  \Upsilon \Psi^{-1} ) \rho(k) - \kappa (q - 1) \Upsilon^{-1} \Psi  \rho(k^{1/2} f) + \upsilon  (q^{1/2} - q^{-1/2})\Upsilon \Psi^{-1} \rho(e k^{1/2}),
\end{equation}
\begin{equation}
    \A_1^* = h^* + \kappa^* \Upsilon^{-1} \Psi^{-1} \rho(k^{-1}),
\end{equation}
where $h$, $h^*$, $\kappa$, $\kappa^*$ and $\upsilon$ are the following constants:
\begin{gather}
    h = - \frac{q^e - 1}{q-1}, \quad h^* = - q \frac{(q^{N+e-2} + 1)}{q-1}, \\
    \kappa = \frac{q^{e + N/2}}{q-1}, \quad \kappa^* = q^{2-N/2} \frac{(q^{N+e-2} + 1) (q^{N+e-1} + 1)}{(q-1)(q^e + q)}, \quad \upsilon = - \frac{q^{N/2}}{q-1}.
\end{gather}
\end{theorem}

\section{Subspace lattice and quotient graphs}

Here, we recall the definition of the subspace lattice $L_N(q)$ and demonstrate that the two $q$-analogs of the hypercube discussed in the previous sections are related to this graph through a quotient construction.

\subsection{Subspace lattice $L_N(q)$} 
Let $\mathbb{F}_q^N$ refer to a vector space of dimension $N$ defined over the finite field with $q$ elements $\mathbb{F}_q$. Let $\Vq$ be the set of subspaces of $\mathbb{F}_q^N$. Given two subspaces $\U,\V \in \Vq$, it is said that $\U$ \textit{covers} $\V$ if 
\begin{equation}
    \V\subset \U,\qquad \dim{(\U)}=\dim{(\V)}+1.
\end{equation}
The subspace lattice graph $L_N(q)$ has the vertex set $\Vq$. An edge connects two vertices $\U$ and $\V$ if one subspace covers the other. The incidence algebra $\mathcal{I}$ of $L_N(q)$ is generated by matrices that act on the vector space $\mathbb{C}^{|V_q|}$, spanned by orthonormal basis vectors $\ket{\,\U}$ labeled by elements $\U \in V_q$. The set of generators of $\mathcal{I}$ consists of projectors $\mathsf{E}_i^*$ associated to subspaces of dimension $i$,
\begin{equation}\label{def:Ei}
    \bra{\U}\mathsf{E}_i^*\ket{\V} = \left\{\begin{array}{cc}
     1 & \text{dim}(\U) = i,\ \U = \V\\
     0& \text{otherwise,}
\end{array}\right.  
\end{equation}
and the following raising and lowering matrices $\mathsf{R}$ and $\mathsf{L}$,
\begin{gather} \label{def:RL}
       \bra{\U}\mathsf{R}\ket{\V}=\left\{\begin{array}{cc}
     1&  \text{if $\U$ covers $\V$}\\
     0& \text{otherwise}
\end{array}\right. , \qquad
\bra{\U}\mathsf{L}\ket{\V}=\left\{\begin{array}{cc}
     1&  \text{if $\V$ covers $\U$}\\
     0& \text{otherwise.}
\end{array}\right.
\end{gather}
Note that $\mathsf{R}^t=\mathsf{L}$. It is convenient to introduce a diagonal invertible matrix $\mathsf{K}$ that is expressed in terms of the projectors $\mathsf{E}_i^*$, 
\begin{equation} \label{def:K}
\mathsf{K}=\sum_{i=0}^Nq^{\frac{N}{2}-i} \mathsf{E}_i^*.
\end{equation}
It is shown in \cite{Bernard_dual, Watanabe} that there exists a surjective algebra homomorphism between $\mathsf{R}, \mathsf{L}, \mathsf{K}$ and the generators of $U_{\sqrt{q}}(\mathfrak{su}(2))$ that sends
\begin{equation}\label{eq:repsl}
    e\tends q^{\frac{1-N}{4}}\mathsf{L}, \quad f\tends q^{\frac{1-N}{4}}\mathsf{R}, \quad k^{\pm1}\tends\mathsf{K}^{\pm1}.
\end{equation}
By further using the involutive algebra automorphism $\theta$, defined by $\theta(e) = f$, $\theta(f) = e$, and $\theta(k^{\pm 1}) = k^{\mp 1}$, one obtains a second homomorphism:
\begin{equation}\label{eq:repsl2}
    e\tends q^{\frac{1-N}{4}}\mathsf{R}, \quad f\tends q^{\frac{1-N}{4}}\mathsf{L}, \quad k^{\pm1}\tends\mathsf{K}^{\mp1}.
\end{equation}
The twisted primitive element $Y = \left(\sqrt{q} f + \frac{1}{\sqrt{q}} e\right)k^{-1/2}$ of $U_q(\mathfrak{su}(2))$, introduced in Section \ref{sec:2}, is mapped under the first homomorphism \eqref{eq:repsl} and a similarity transformation to the weighted adjacency matrix $M_q(N)$ introduced in \cite{ghosh2022q}, describing a $q$-analog of the hypercube based on the subspace lattice. This matrix was further studied in \cite{terwilliger2023aq, terwilliger2024projective} for its $Q$-polynomial structure, similar to that of the $N$-cube. The image of $Y$ under the second homomorphism \eqref{eq:repsl2} will play a key role in connecting the symplectic dual polar graph to the weighted hypercube, and will be denoted
\begin{equation}\label{def:Ysf}
    \mathsf{Y} = q^{\frac{1-N}{4}}( q^{1/4} \mathsf{L} + q^{-1/4} \mathsf{R}) \mathsf{K}^{1/2}.
\end{equation}
It is readily observed that this corresponds to the adjacency matrix of a weighted subspace lattice graph.

Next, we demonstrate the existence of a $U_{\sqrt{q}}(\mathfrak{su}(2))$-submodule within the standard module of the subspace lattice $\mathbb{C}^{|\Vq|}$, that is linked to the adjacency algebra of the weighted hypercube from Section \ref{sec:2}.


\subsection{Relation to weighted cubes} To establish a connection between the subspace lattice $L_N(q)$ and the weighted $N$-cube of Section \ref{sec:2}, we require a map from the set of subspaces $V_q$ to the set of vertices of the hypercube $V = \{0,1\}^N$. First, we consider an injective map $\tau$ between $V_q$ and $N\times N$ matrices with entries in $\mathbb{F}_q$.

\begin{definition}\label{def:tau}
    Let $\{e_1, e_2, \dots e_N\}$ be a basis of $\mathbb{F}_q^N$. Let $V_q$ denote the set of subspaces of $\mathbb{F}_q^N$. The map $\tau: V_q \rightarrow \text{End}(\mathbb{F}_q^N)$ is defined as the map taking a subspace $\V \in V_q$ to the unique upper triangular $N\times N$ matrix $\boldsymbol{v}$ with entries $\boldsymbol{v}_{ij} \in \mathbb{F}_q$, whose columns $\boldsymbol{v}_j$ span $\V$,
\begin{equation}
     \V = \text{span}\left\{ \boldsymbol{v}_j = \sum_i v_{ij} e_i \ | \ j = 1,2,\dots N\right\},
\end{equation}
and whose diagonal entries verify $\boldsymbol{v}_{ii} \in \{0,1\}$ and
\begin{equation}
    \boldsymbol{v}_{ij} \neq 0\quad  \Rightarrow  \quad \boldsymbol{v}_{ii} \neq 1, \ \boldsymbol{v}_{jj} \neq 0.
\end{equation}
\end{definition}
The upper triangular matrices $\boldsymbol{v} = \tau(\V)$ in the previous definition have diagonal entries restricted to $0$ and $1$. A column $\boldsymbol{v}_j$ in $\boldsymbol{v}$ consists entirely of $0$ if $\boldsymbol{v}_{jj} = 0$. Furthermore, when $\boldsymbol{v}_{jj} = 1$, the $j$-th column is the only one with a non-zero entry in the $j$-th row. For instance, for $N = 5$, the possible matrices that satisfy these conditions and have $(0,1,0,1,1)$ on the diagonal are
\begin{equation}
    \boldsymbol{v} = \begin{pmatrix}
        0 & \boldsymbol{v}_{12} & 0 & \boldsymbol{v}_{14} & \boldsymbol{v}_{15}\\
        0 & 1 & 0 & 0 & 0\\
        0 &0 & 0 & \boldsymbol{v}_{34} & \boldsymbol{v}_{35}\\
        0 & 0 & 0 &1 & 0\\
        0 & 0 & 0 & 0 & 1\\
    \end{pmatrix}.
\end{equation}
It follows from a Gaussian elimination procedure that, for any subspace 
$\V \in V_q$, there exists a unique matrix $\boldsymbol{v}$ satisfying the conditions in Definition \ref{def:tau}, ensuring that $\tau$ is well-defined. The next lemma translates the notion of covering for subspaces in 
$V_q$ into constraints on their image under $\tau$ in $\text{End}(\mathbb{F}_q^N)$.

\begin{lemma} \label{lem:corr}
    Let $\V$ and $\U$ be subspaces of the vector space $\mathbb{F}_q^N$, with $\boldsymbol{v} = \tau(\V)$ and $\boldsymbol{u} = \tau(\U)$ their corresponding upper triangular matrices under the map $\tau$ from Definition \ref{def:tau}. Let $\boldsymbol{v}_{j}$ and $\boldsymbol{u}_{j}$ denote the $j$-th column of the matrices $\boldsymbol{v}$ and $\boldsymbol{u}$ respectively. The subspace $\V$ covers $\U$ if and only if the following conditions are satisfied:
    \begin{enumerate}
        \item[(i)]  The diagonal of $\boldsymbol{v}$ and $\boldsymbol{u}$ are at Hamming distance $1$, with a unique coordinate $k \in \{1,2,\dots, N\}$ such that $\boldsymbol{v}_{kk} = 1$ and $\boldsymbol{u}_{kk} = 0$.

        \item[(ii)]  We have $ \boldsymbol{v}_{j} =  \boldsymbol{u}_{j}$ for all $j = 1,2,\dots , k-1$. 
        \item[(iii)] There exist coefficients $c_j \in \mathbb{F}_q$ such that $ \boldsymbol{u}_{j} =  \boldsymbol{v}_{j} + c_j \boldsymbol{v}_k$ for all $j = k+1,k+2,\dots , N$.  
    \end{enumerate}
\end{lemma}
\begin{proof}
The derivation is based on simple concepts of linear algebra and is provided in Appendix \ref{app:A}.
\end{proof}

\begin{corollary}\label{cor:nk}
    Let $x$ and $y$ be sequences in $\{0,1\}^N$ at Hamming distance $1$, with $x_k = 1$ and $y_k = 0$. Given any subspace $\mathcal{V} \subseteq \mathbb{F}_q^N$ whose matrix $\tau(\V) = \boldsymbol{v}$ has diagonal entries given by $x$, the number $n_k(x)$ of subspaces $\U$ covered by $\V$ and whose matrix $\tau(\U) = \boldsymbol{u}$ has entries on the diagonal given by $y$ is
    \begin{equation}
        n_k(x) = q^{\sum_{j = k+1}^N x_j}.
    \end{equation}
\end{corollary}
\begin{proof}
    The subspaces $\U$ covered by $\V$ whose matrix $\tau(\U) = \boldsymbol{u}$ has entries on the diagonal given by $y$ are in correspondence with the different possible coefficients $c_j$ in condition (iii) of Lemma \ref{lem:corr}, with $j$ such that $x_j = \boldsymbol{v}_{jj} \neq 0$. The result follows from a counting argument.
\end{proof}

Now, let us denote $\tau_{diag} : V_q \rightarrow V$ the surjective map that takes a subspace $\V$ to the $(0,1)$-sequence of length $N$ corresponding to the diagonal of $\tau\left(\V\right)$,
\begin{equation}
    \tau_{diag}\left(\V\right) = (\boldsymbol{v}_{11}, \boldsymbol{v}_{22}, \dots \boldsymbol{v}_{NN}).
\end{equation}
Let $\tau_{diag}^{-1}(x)$ be the set of subspaces $\V$ whose associated matrix $\tau(\V) = \boldsymbol{v}$ has the sequence $x \in V$ as diagonal $(\boldsymbol{v}_{11}, \boldsymbol{v}_{22}, \dots \boldsymbol{v}_{NN})$, and $|\tau_{diag}^{-1}(x)|$ the size of that set. This can be readily expressed as
\begin{equation}\label{eq:size}
    |\tau_{diag}^{-1}(x)| = q^{\sum_{i=1}^N \sum_{j = 1}^{i-1} x_i (1-x_j)}.
\end{equation}
We use $\tau_{diag}$ and its preimage to define a linear map $\phi: \mathbb{C}^{|V|} \rightarrow \mathbb{C}^{|V_q|} $ from the standard module of the hypercube to the standard module of the subspace lattice as follows. 

\begin{definition}\label{def:phi}
 The linear map $\map: \mathbb{C}^{|V|} \rightarrow \mathbb{C}^{|V_q|} $ acts on the basis vectors $\ket{x} \in \mathbb{C}^{|V|}$, with $x \in V  = \{0,1\}^N$, as follows
\begin{equation}
    \map(\ket{x}) = \frac{1}{\sqrt{|\tau_{diag}^{-1}(x)|}}\sum_{\V \in \tau_{diag}^{-1}(x)} \ket{\V}.
\end{equation}
\end{definition}
In other words, $\map$ takes the basis vectors $\ket{x}$ of the standard module of the hypercube to a coherent sum of vectors $\ket{\V}$ in the standard module of the subspace lattice. The sum is over the subspaces $\V$ such that $\tau_{diag}(\V) = x$, and each vector $\map(\ket{x})$ is the characteristic vector of an equivalence class in $V_q/\!\!\sim$, with $\V \sim \V'$ if and only if $\tau_{diag}(\V) =\tau_{diag}(\V')$. The following lemma describes the action of the generators $\mathsf{R}$, $\mathsf{L}$ and $\mathsf{E}_i^*$ of the incidence algebra of $L_N(q)$ on the vectors $\map(\ket{x})$.

\begin{lemma}\label{lem:actRL}
    Let $\mathsf{R}$, $\mathsf{L}$ and $\mathsf{E}_i^*$ be the generators of the incidence algebra of $L_N(q)$, as defined in \eqref{def:Ei}-\eqref{def:RL}. Let $\map: \mathbb{C}^{|V|} \rightarrow \mathbb{C}^{|V_q|} $ be the linear map of Definition \eqref{def:phi} and $\hat{k}$ denote the $(0,1)$-sequence of length $N$ with a single non-zero entry in the $k$-th position. For any $(0,1)$-sequence $x$ of length $N$, one finds
    \begin{equation}\label{eq:R}
        \mathsf{R}\circ \map (\ket{x}) = \sum_{\substack{k \in \{1,2,\dots, N\} \\ x_k = 0} } q^{\frac{k-1}{2} - \frac{1}{2}\sum_{j = 1}^{k-1} x_j + \frac{1}{2}\sum_{j = k+1}^N x_j} \map\left(\ket{x + \hat{k}}\right),
    \end{equation}
        \begin{equation}\label{eq:L}
        \mathsf{L} \circ \map ( \ket{x}) = \sum_{\substack{k \in \{1,2,\dots, N\} \\ x_k = 1} } q^{\frac{k-1}{2} - \frac{1}{2}\sum_{j = 1}^{k-1} x_j + \frac{1}{2}\sum_{j = k+1}^N x_j} \map\left(\ket{x - \hat{k}}\right),
    \end{equation}
       \begin{equation}\label{eq:ei}
        \mathsf{E}_i^* \circ \map ( \ket{x}) =  \left\{
	\begin{array}{ll}
		\map(\ket{x}) & \mbox{if } i = \sum_{j = 1}^N x_j \\
		0 & \mbox{otherwise. }
	\end{array}
\right.
    \end{equation}
\end{lemma}
\begin{proof}
    The vector $\map(\ket{x})$ is a coherent sum over vectors $\ket{\, \U}$ with $\U \in V_q$ such that $\tau_{diag}(\U) = x$. It follows from the definition of $\mathsf{R}$ and Lemma \ref{lem:corr} that 
    \begin{equation}
        \mathsf{R} \circ \map(\ket{x}) \in \text{span}\{\ket{\V}  \ | \ \V \in V_q \ \text{s.t.} \ \tau_{diag}(\V) = x + \hat{k} \  \text{where}\  x_k = 0  \}
    \end{equation}
    Let $\V$ be a subspace such that $\tau_{diag}(\V) = x + \hat{k}$ with $k$ such that $x_k = 0$. From Corollary \ref{cor:nk}, there are $q^{\sum_{j = k+1}^N x_j}$ subspaces $\U$ covered by $\V$ such that $\tau_{diag}(\U) = x$ and thus
    \begin{equation}
        \bra{\V}( \mathsf{R} \circ \map )\ket{x} =\frac{1}{\sqrt{|\tau_{diag}^{-1}(x)|}} \sum_{\U \in \tau_{diag}^{-1}(x)} \bra{\V}\mathsf{R}\ket{\, \U} = \frac{q^{\sum_{j = k+1}^N x_j}}{\sqrt{|\tau_{diag}^{-1}(x)|}}.
    \end{equation}
    This leads to \eqref{eq:R} upon using equation \eqref{eq:size} to obtain the following ratio:
    \begin{equation}
        \frac{\sqrt{|\tau_{diag}^{-1}(x + \hat{k})|}}{\sqrt{|\tau_{diag}^{-1}(x)|}} = q^{\frac{k-1}{2}-\frac{1}{2}\sum_{j = 1}^{N}x_j}.
    \end{equation}
    Equation \eqref{eq:L}, describing the action of $\mathsf{L}$, is derived by identifying $\mathsf{L}$ as the transpose of $\mathsf{R}$. Lastly, $\mathsf{E}_i^*$ acts diagonally on $\map(\ket{x})$, as described in \eqref{eq:ei}, since all vectors $\ket{\V}$ in the coherent sum defining $\map(\ket{x})$ correspond to a subspace $\V$ of dimension $\sum_{i=1}^N x_i$. Indeed, this follows from $\tau_{diag}(\V) = x$ and the correspondence between the dimension of the subspace and the number of non-zero entries on the diagonal of its matrix representation under $\tau$.
\end{proof}

\begin{corollary}
    Let $\mathcal{I}$ be the incidence algebra of the subspace lattice $L_N(q)$. Let $\map: \mathbb{C}^{|V|} \rightarrow \mathbb{C}^{|V_q|} $ be the linear map of Definition \eqref{def:phi}. Then $\map(\mathbb{C}^{|V|})$ forms a $\mathcal{I}$-submodule of the standard module $\mathbb{C}^{|V_q|}$ of the subspace lattice $L_N(q)$. 
\end{corollary}
\begin{proof}
   From the previous lemma, one observes that the action of the generators of $\mathcal{I}$ is closed on the basis vector $\map(\ket{x})$ of $\map(\mathbb{C}^{|V|})$. The result follows.
\end{proof}

We now show the main result of this subsection, namely that the quotient graph of a weighted subspace lattice graph $L_N(q)$ (with respect to the subsets of vertices defined from the preimage of $\tau_{diag}$) is a the weighted hypercube.

\begin{proposition}\label{prop:wc}
    Let $\map: \mathbb{C}^{|V|} \rightarrow \mathbb{C}^{|V_q|} $ be the linear injective map of Definition \eqref{def:phi}. Let $\mathsf{Y}$ be the element of the incidence algebra of $L_N(q)$ defined in \eqref{def:Ysf}. Let $A_q$ be the adjacency matrix of the weighted cube defined in \eqref{A_q_def}. Then
    \begin{equation}\label{eq:reswc}
       \mathsf{Y} \circ \map =  \map \circ \left( \pi^{-1} \circ A_{1/\sqrt{q}}\circ  \pi \right)
    \end{equation}
    with $\pi : \mathbb{C}^{|V|} \rightarrow \mathbb{C}^{|V|}$ an automorphism of the hypercube whose action amounts to reversing the sequences $x \in \{0,1\}^N$, i.e.
    \begin{equation}
        \pi \ket{x_1,x_2, \dots, x_N} = \ket{x_N,x_{N-1}, \dots, x_1}.
    \end{equation}
\end{proposition}
\begin{proof}
    Let $x$ and $y$ be two sequences in $\{0,1\}^N$. Recall that $\mathsf{Y} =  q^{\frac{1-N}{4}}( q^{1/4} \mathsf{L} + q^{-1/4} \mathsf{R}) \mathsf{K}^{1/2}$. From the action of $\mathsf{R}$, $\mathsf{L}$ and $\mathsf{E}_i^*$ given in Lemma \eqref{lem:actRL}, we find that
\begin{equation}\label{eq:plhs}
      \bra{y}  \map^{-1}\circ \left(  ( q^{1/4} \mathsf{L} + q^{-1/4} \mathsf{R}) \mathsf{K}^{1/2} \right) \circ \map \ket{x}=\left\{ \begin{array}{cc}
       q^{\frac{N-1}{4} + \frac{k-1}{2} - \sum_{j = 1}^{k-1} x_j } & \text{if } \partial(x,y) = 1, \ x_{k} \neq y_{k} \\
        0 & \text{otherwise,} 
    \end{array}\right.
\end{equation}
where $\map^{-1}$ is the inverse of $\map$ defined on the domain $\map(\mathbb{C}^{|V_q|})$. Furthermore, equation \eqref{def:aq} gives the entries of $A_{1/\sqrt{q}}$:
\begin{equation}
    \bra{y}A_{1/\sqrt{q}}\ket{x}=\left\{ \begin{array}{cc}
          q^{\frac{N-k}{2}  - \sum_{j =k+1}^{N} x_j }& \text{if } \partial(x,y) = 1, \ x_{k} \neq y_{k} \\
        0 & \text{otherwise.} 
    \end{array}\right.
\end{equation}
Since $x_i \rightarrow x_{N-i+1}$ under $\pi$, we have
\begin{equation}\label{eq:prhs}
    \bra{y} \pi^{-1} \circ A_{1/\sqrt{q}}\circ \pi \ket{x}=\left\{ \begin{array}{cc}
          q^{\frac{k-1}{2}  - \sum_{j =1}^{k-1} x_j}& \text{if } \partial(x,y) = 1, \ x_{k} \neq y_{k} \\
        0 & \text{otherwise.} 
    \end{array}\right.
\end{equation}
The correspondence between equations \eqref{eq:prhs} and \eqref{eq:plhs} yields the results.
    
\end{proof}

The next subsection demonstrates that the weighted subspace lattice with adjacency matrix $ \mathsf{Y} =q^{\frac{1-N}{4}}( q^{1/4} \mathsf{L} + q^{-1/4} \mathsf{R}) \mathsf{K}^{1/2}$ also corresponds to the quotient graph of dual polar graphs of type $C_d(q)$.

\subsection{Relation to dual polar spaces}
It was shown in \cite{Bernard_dual} that weighted subspace lattice graphs can be obtained as quotient graphs of the symplectic dual polar graph, i.e. the dual polar graph of type $C_d(q)$. Specifically, there exists a change of basis that decomposes the adjacency matrix of the symplectic dual polar graph into a direct sum of adjacency matrices of weighted subspace lattice graphs. Here, we recall some key details regarding this result.

Let $\text{Sym}_d$ denote the set of $d \times d$ symmetric matrices with entries in $\mathbb{F}_q$. We define the type $\epsilon$ of a matrix $S \in \text{Sym}_d$ as follows:
\begin{definition}
    \normalfont The \textit{type} $\epsilon$ of a symmetric matrix $S \in \text{Sym}_d$ is defined as:
    \begin{equation}
        \epsilon = \begin{cases}
            1 & \text{if rank($S$) is 0 \textbf{or} is even and } (-1)^{\frac{\text{rank}(S)}{2}} \text{det}(Q) \text{ is a square in } \mathbb{F}_q, \\
            -1 & \text{if rank($S$) is even and } (-1)^{\frac{\text{rank}(S)}{2}} \text{det}(Q) \text{ is a non-square in } \mathbb{F}_q, \\
            0 & \text{if rank($S$) is odd,}
        \end{cases}
    \end{equation}
where $Q$ is a rank($S)\times$ rank($S$) matrix such that, for a matrix $\Upsilon\in GL(d,q)$, we have
\begin{equation}
    \Upsilon^tS\Upsilon=\begin{pmatrix}
    0&0\\0&Q
\end{pmatrix}.
\end{equation}
\end{definition}

\begin{theorem} (Reformulation of Theorem 7.1 of \cite{Bernard_dual})
    Let $\A_1$ denote the adjacency matrix of the dual polar graph of type $C_d(q)$, defined on the set $X$ of maximal isotropic subspaces of the vector space $\mathbb{F}_q^{2d}$ equipped with a non-degenerate symplectic form $\mathfrak{B}$. The standard module $\mathbb{C}^{|X|}$ of this graph decomposes as a direct sum of submodules $W(S)$ invariant under the action of $\A_1$, labeled by symmetric matrices $S \in \text{Sym}_d$,
    \begin{equation}\label{eq:dec}
    \mathbb{C}^{|X|} = \bigoplus_{S \in \text{Sym}_d} W(S), \quad \A_1 W(S) \subseteq W(S).
\end{equation}
Furthermore, for each submodule $W(S)$ there exists a linear bijective map $\phi: \mathbb{C}^{|V_q|}  \rightarrow W(S)$ between the subspaces $W(S)$ and the standard module of the subspace lattice $L_{d-\text{rank}(S)}(q)$ such that
\begin{equation}\label{eq:rws}
    \A_1|_{W(S)} = \phi^{-1} \circ \left( \epsilon q^{d/2}{\mathsf{K}} - 1 + q^{d/2} \mathsf{Y}\right) \circ \phi
\end{equation}
with $\epsilon$ the type of $S$ and $\mathsf{Y}$, $\mathsf{K}$ the elements of the incidence algebra of $L_
{d-\text{rank}(S)}(q)$ as defined in \eqref{def:Ysf} and \eqref{def:K}.
\end{theorem}

\begin{proof}
    The detailed proof can be found in \cite{Bernard_dual}; here, we summarize the key ideas. The group of Lie type $Sp(2d, q)$ preserves the form $\mathfrak{B}$ and naturally acts on the vertices of the symplectic dual polar graph, forming a subgroup of its automorphism group. Since graph automorphisms are linked to the commutant of the adjacency algebra, it follows that the adjacency matrix has a closed action on the eigenspaces of an abelian subgroup of automorphisms. The decomposition \eqref{eq:dec} arises from this principle, applied to an abelian subgroup $H \cong (\mathbb{F}_q, +)^{d(d+1)/2}$, which stabilizes a subspace $x_0 \in X$.

The derivation of equation \eqref{eq:rws}, which describes the restriction of the adjacency matrix $\A_1$ to each eigenspace of the abelian subgroup $H$, proceeds in two steps. First, by decomposing the stabilizer of $x_0$ in $Sp(2d, q)$ as a semi-direct product of $H$ and $GL(d, q)$, and using a projective representation approach, one identifies a basis of $W(S)$ in bijection with the set of subspaces of $\mathbb{F}_q^{d - \text{rank}(S)}$. Second, through combinatorial arguments, the action of $\A_1$ on this basis is matched with the action of the generators of the incidence algebra of $L_{d - \text{rank}(S)}(q)$, leading to the desired result.
\end{proof}

Since Proposition \ref{prop:wc} established that the weighted subspace lattice 
$L_N(q)$, with adjacency matrix $\mathsf{Y} = q^{\frac{1-N}{4}} ( q^{1/4} \mathsf{L} + q^{-1/4} \mathsf{R}) \mathsf{K}^{1/2}$, has the weighted cube as its quotient graph, it follows from the previous theorem that the adjacency matrix of the symplectic dual polar graph restricted to a subspace $W(S)$ with $\epsilon = 0$ can be further restricted to a subspace where it acts as an affine transformation of the adjacency matrix of the weighted cube.

\section{Outlook}

This work established a connection between two different approaches to the $q$-deformation of hypercubes: the weighted cube and the symplectic dual polar graph. Specifically, it was shown that the adjacency matrix of the symplectic dual polar graph decomposes into a direct sum of adjacency matrices of weighted subspace lattices, which in turn have the weighted hypercube as a quotient graph. This clarifies the relation between the various definitions of the $q$-hypercube found in the literature and highlights the central role played by a twisted primitive element of $U_q(\mathfrak{su}(2))$ in each approach.

The adjacency matrix of the weighted cube appears in various physical contexts. It is closely related to the $q$-Dicke states studied in quantum information theory, and whose quantum entanglement has been extensively studied in \cite{li2015entanglement, raveh2024q}. The matrix $A_q$, as the representation of an element of $U_q(\mathfrak{su}(2))$ in the $N$-fold tensor product of the fundamental representation, also happens to be a generator of the symmetry algebra of certain spin chains with integrable boundary conditions \cite{kulish1991general, pasquier1990common}. Given the connection between $A_q$ and the adjacency matrix of symplectic the dual polar graph established in this paper, it becomes particularly intriguing to explore how the combinatorial properties of these association schemes can be applied to these problems. In future work, we also aim to extend the study of entanglement in free fermion systems on hypercubes \cite{Bernard_Hamming} to the various $q$-analogs considered here.

\section*{Acknowledgments}

The authors are grateful to Wolter Groenevelt, Erik Koelink, Hjalmar Rosengren and Jasper V. Stokman for their invitation to contribute to a special volume for Tom Koornwinder. PAB holds an Alexander-Graham-Bell scholarship from the Natural Sciences and Engineering Research Council (NSERC) of Canada. EP held a NSERC Undergraduate Student Research Award (USRA) in the course of this project. LV is funded in part through a discovery grant from NSERC.

\appendix
\section{Derivation of Lemma \ref{lem:corr}} \label{app:A}

First, we introduce two lemmas concerning the properties of the matrices in the image of $\tau$, which will be useful in the derivation of Lemma \ref{lem:corr}.

\begin{lemma}\label{lem:sub}
      Let $\V$ and $\U$ be subspaces of the vector space $\mathbb{F}_q^N$, with $\boldsymbol{v} = \tau(\V)$ and $\boldsymbol{u} = \tau(\U)$ their corresponding upper triangular matrices under the map $\tau$. If $\U \subseteq \V$, then
     \begin{equation}
         \boldsymbol{u}_{ii} = 1 \quad \Rightarrow \quad  \boldsymbol{v}_{ii} = 1.
     \end{equation}
\end{lemma}
\begin{proof}
    If $\U$ is a subspace of $\V$, then each column of $\boldsymbol{u} = \tau(\U)$ lies in the span of the columns of $\boldsymbol{v} =\tau(\V)$. By definition of the map $\tau$, a column of $\boldsymbol{v}$ can have non-zero $i$-th entry only if $\boldsymbol{v}_{ii} = 1$. The same condition applies to any linear combination of these columns, leading to the desired result.
\end{proof}

\begin{lemma}\label{lem:col}
     Let $\V$ be a subspace of the vector space $\mathbb{F}_q^N$, with $\boldsymbol{v} = \tau(\V)$ the corresponding upper triangular matrix under the map $\tau$. Let $\boldsymbol{v}_{j}$ denote the $j$-th column of $\boldsymbol{v}$. For any given coefficients $c_j \in \mathbb{F}_q$, we have
     \begin{equation}
       \boldsymbol{v}_{jj} = 1 \quad \Rightarrow \quad   \left(\sum_{i = 1}^N c_{i} \boldsymbol{v}_{i}\right)_j = c_{j}
     \end{equation}
\end{lemma}
\begin{proof}
   Due to the structure of the upper triangular matrices in the image of $\tau$, we have $\boldsymbol{v}_{jj} = 1$ if and only if $\boldsymbol{v}_j$ is the only column with a non-zero entry in the $j$-th position. Consequently, the $j$-th entry of any linear combination of these columns will solely depend on $\boldsymbol{v}_j$. This proves the result.
\end{proof}

\noindent We are now ready for the derivation of Lemma \ref{lem:corr}.

\begin{proof} (Lemma \ref{lem:corr})
A subspace $\V$ is said to cover $\U$ if and only if $\U \subset \V$ and $\dim(\V) = \dim(\U) + 1$. Notice that the number of non-zero entries on the diagonals of $\boldsymbol{u}$ and $\boldsymbol{v}$ corresponds to the dimensions of $\U$ and $\V$, respectively. Therefore, if condition (i) holds, it directly follows that $\dim(\V) = \dim(\U) + 1$. Furthermore, it is straightforward to verify that conditions (ii) and (iii) imply that the non-zero columns of $\boldsymbol{u}$ are linear combinations of those of $\boldsymbol{v}$, which ensures that $\U \subset \V$. Thus, the combination of conditions (i), (ii) and (iii) guarantees that $\V$ covers $\U$.

Now let us assume that $\V$ covers $\U$. It follows from $\U \subset \V$ and Lemma \ref{lem:sub} that $\boldsymbol{u}_{ii} = 1 \Rightarrow \boldsymbol{v}_{ii} = 1 $. Since the number of non-zero entries on the diagonal corresponds to the dimension of the subspace and $\dim(\V) = \dim(\U) + 1$, the diagonal of $\boldsymbol{v}$ has one non-zero entry more than $\boldsymbol{u}$. Condition (i) follows. Next, we have to show that $\U \subset \V$ implies (ii) and (iii). For all $j \neq k$ such that $\boldsymbol{u}_{jj} = 0$, we have  $\boldsymbol{v}_{jj} = 0$ and $\boldsymbol{u}_j = \boldsymbol{v}_j = 0$, so that both (ii) and (iii) are satisfied. For $j$ with $\boldsymbol{u}_{jj} \neq 0$, it follows from $\U \subset \V$ that there exist coefficients $c_{ij}$ such that
\begin{equation}\label{eq:p1}
    \boldsymbol{u}_j = \sum_{i = 1}^N c_{ij} \boldsymbol{v}_i, \quad c_{ij} \in \mathbb{F}_q.
\end{equation}
For $i$ such that $\boldsymbol{v}_{ii} = 0$, the structure of $\boldsymbol{v}$ implies that the column $\boldsymbol{v}_{i}$ is the zero vector and we can fix without loss of generality $c_{ij} =0$. For $i \neq k$ such that $\boldsymbol{v}_{ii} = 1$, Lemma \ref{lem:col} implies that $i$-th entry of l.h.s. of \eqref{eq:p1} is $c_{ij}$. We also have that $i \neq k$ and $\boldsymbol{v}_{ii} = 1$ implies $\boldsymbol{u}_{ii} = 1$, so that the $i$-th entry of the r.h.s is $\boldsymbol{u}_{ij}  =0$. The equality of \eqref{eq:p1} then implies $ c_{ij} = 0$ for all $i \notin \{j, k\}$, leading to
\begin{equation}
    \boldsymbol{u}_j = c_{jj} \boldsymbol{v}_j + c_{kj} \boldsymbol{v}_k. 
\end{equation}
Since $\boldsymbol{u}_{jj} = \boldsymbol{v}_{jj} = 1$ and $\boldsymbol{v}_{jk} = 0$, it follows that $c_{jj} = 1$. This leads to condition (iii). In the case $j < k$, the upper triangular structure imposes $\boldsymbol{u}_{kj} = \boldsymbol{v}_{kj} = 0$ and so $c_{kj} =0$, leading to (ii).

\end{proof}

\bibliographystyle{plain}

\end{document}